\documentclass[a4paper]{amsart}

\usepackage{amsfonts}
\usepackage[colorlinks=true,citecolor=blue]{hyperref}
\usepackage{mathptmx}
\usepackage{eucal}
\usepackage{graphicx}
\usepackage{mathrsfs}
\usepackage{amssymb}
\usepackage{amsmath}
\usepackage{amsthm}
\usepackage{enumerate}
\usepackage{verbatim}

\usepackage{xcolor}
\newtheorem{thm}{Theorem}[section]
\newtheorem{cor}[thm]{Corollary}
\newtheorem{que}[thm]{Question}
\newtheorem{lem}[thm]{Lemma}
\newtheorem{prop}[thm]{Proposition}
\newtheorem{conj}[thm]{Conjecture}
\newtheorem*{conjecure}{Conjecture}

\theoremstyle{definition}
\newtheorem{defn}[thm]{Definition}
\newtheorem{exam}[thm]{Example}

\numberwithin{equation}{section}

\newcommand{\eps}{\varepsilon}
\newcommand{\N}{\mathbb{N}}
\newcommand{\Z}{\mathbb{Z}}
\newcommand{\Orb}{\mathrm{Orb}}

\DeclareMathOperator{\diam}{diam}
\DeclareMathOperator{\Int}{int}

\begin{document}
\title{On dynamics of graph maps with zero  topological entropy}
\author[J. Li]{Jian Li}
\address[J. Li]{Department of Mathematics,
Shantou University, Shantou, 515063, Guangdong, China}
\email{lijian09@mail.ustc.edu.cn}

\author[P. Oprocha]{Piotr Oprocha}
\address[P. Oprocha]{AGH University of Science and Technology, Faculty of Applied
    Mathematics, al.
    Mickiewicza 30, 30-059 Krak\'ow, Poland
    -- and --
    National Supercomputing Centre IT4Innovations, Division of the University of Ostrava,
    Institute for Research and Applications of Fuzzy Modeling,
    30. dubna 22, 70103 Ostrava,
    Czech Republic}
\email{oprocha@agh.edu.pl}

\author[Y. Yang]{Yini Yang}
\address[Y. Yang]{Department of Mathematics,
Shantou University, Shantou, 515063, Guangdong, China}
\email{ynyangchs@foxmail.com}

\author[T.~Zeng]{Tiaoying Zeng}
\address[T.~Zeng]{Department of Mathematics, Shantou University, Shantou, Guangdong 515063, P.R. China
-- \and -- School of Mathematics, Jiaying University, Meizhou, Guangdong  514015, P.R. China }
\email{zengtiaoying@126.com}

\date{\today}

\begin{abstract}
We explore the dynamics of graph maps with zero topological entropy.
It is shown that a continuous map $f$ on a topological graph $G$
has zero topological entropy if and only if it is locally mean equicontinuous, that is
the dynamics on each orbit closure is mean equicontinuous.
As an application, we show that Sarnak's M\"obius Disjointness Conjecture
is true for  graph maps with zero topological entropy.
We also extend several results known in interval dynamics to graph maps.
We show that a graph map has zero topological entropy if and only if there is no $3$-scrambled tuple
if and only if the proximal relation is an equivalence relation;
a graph map has no scrambled pairs if and only if it is null
if and only if it is tame.
\end{abstract}
\maketitle

\section{Introduction}
Let $(X,f)$ be a topological dynamical system and $(c_n)_{n=1}^\infty$ a sequence of complex numbers.
We say that the sequence $(c_n)_{n=1}^\infty$ is \emph{linearly disjoint} from $(X,f)$ if
for any $x\in X$ and any continuous function $\varphi\colon X\to\mathbb{C}$,
\[\lim_{N\to\infty}\frac{1}{N}\sum_{n=1}^N c_n\varphi(f^n(x))=0.\]

An interesting sequence is obtained by evaluation of the \textit{M\"obius function}
$\mu\colon \N\to \{-1,0,1\}$ which is defined as follows:
$\mu(1)=1$, $\mu(n) = (-1)^k$ when $n$ is a product of $k$ distinct primes and $\mu(n) = 0$
otherwise.
In \cite{S09}, Sarnak stated the following conjecture.

\begin{conjecure}[M\"obius Disjointness Conjecture]
	The M\"obius function $\mu(n)$ is linearly
	disjoint from all dynamical systems with zero topological entropy.
\end{conjecure}

The case when $(X, f)$ is a finite periodic orbit is covered by the
Prime Number Theorem, which says that the number of
prime numbers less than or equal to $N$ is asymptotically $N/\ln N$ (e.g. see \cite{HW}).
The case when $f$ is a rotation on the circle is covered by Davenport's theorem which says that (see \cite{Dav}):
$$
\lim_{N\to\infty}\frac{1}{N}\sum_{n=1}^N \mu(n)e^{2\pi i \alpha n}=0
$$
for every $\alpha \in [0,1)$.
Recently, many other special cases of M\"obius Disjointness Conjecture have been considered.
In particular, Karagulyan proved in \cite{K15} that M\"obius Disjointness Conjecture
holds for continuous interval maps with zero topological entropy
and orientation-preserving circle homeomorphisms.

In \cite{FJ15}, Fan and Jiang introduced minimally mean attractable flows
and minimal mean Lyapunov stable flows
and proved that every oscillating sequence (which includes sequence defined by the M\"obius function)
is linearly disjoint from all dynamical systems which are minimally mean attractable and
minimal mean Lyapunov stable.
They also provided several examples of minimally mean attractable and minimally mean Lyapunov stable
flows, including all Feigenbaum zero topological
entropy flows, and all orientation-preserving circle homeomorphisms.
Motivated by the result of Karagulyan in \cite{K15},
Fan and Jiang conjectured in \cite[Remark 8]{FJ15} that
all continuous interval maps with
zero topological entropy are minimally mean attractable and minimally mean Lyapunov stable.

One of the motivations of this paper is to confirm this conjecture.
In fact, we show a little more, proving that every graph map with zero topological
entropy is locally mean equicontinuous, which in turn implies
 minimally mean attractability and minimally mean Lyapunov stability.
Strictly speaking, we have the following result.

\begin{thm}\label{thm:main1}
Let $f\colon G\to G$ be a graph map.
Then $f$ has zero topological entropy
if and only if it is locally  mean equicontinuous, that is
for every $x\in G$,
$(\overline{Orb(x,f)},f)$ is mean equicontinuous.
\end{thm}

Combining this result with \cite{FJ15} we obtain the following (see Section~\ref{sec:SMDC} for more details):

\begin{thm}\label{thm:main2}
Any oscillating sequence is linearly disjoint from all continuous graph maps with zero topological entropy.
In particular, Sarnak's M\"obius Disjointness Conjecture holds for graph maps with zero topological entropy.
\end{thm}

It should be noticed that there is another approach to Theorem~\ref{thm:main2}.
Recently in \cite{HWY16} Huang, Wang and Zhang
provided a criterion for Sarnak's M\"obius Disjointness Conjecture,
which includes locally mean equicontinuous systems.
By Theorem~\ref{thm:main1}, every graph map with zero topological entropy
is locally mean equicontinuous and then satisfies the criterion.
But it is also worth mentioning at this point that results of \cite{HWY16}
heavily rely on combinatorial properties of M\"obius function and hence can be applied only to this function.

As a related topic, we characterize graph maps with zero topological entropy
in terms of scrambled pairs and tuples.
In \cite[Theorem 4.21]{L11} the author proved that if
an interval map  has zero topological entropy,
then it does not have any scrambled $3$-tuples.
A careful analysis of the structure of $\omega$-limit sets allows us to extend this
result onto the case of graph maps (the proof is presented in Section~\ref{sec:scrambled:entropy}):
\begin{thm}\label{thm:scrambled3tuple}
Let $f\colon G\to G$ be a graph map.
Then $f$ has zero topological entropy
if and only if
for any three pairwise distinct points $x_1,x_2,x_3\in G$,
the tuple $(x_1,x_2,x_3)$ is not scrambled.
\end{thm}

In Proposition 4.3 and Theorem 5.5 of \cite{L11} , the author proved that
an interval map has zero topological entropy if and only if
every proximal pair is Banach proximal if and only if the proximal relation is an equivalence relation.
We also extend this result to graph maps
(the proof is also presented in Section~\ref{sec:scrambled:entropy}):

\begin{thm}\label{thm:proximal}
Let $f\colon G\to G$ be a graph map. The following conditions are equivalent:
	\begin{enumerate}
		\item\label{thm:proximal:1} $f$ has zero topological entropy,
		\item\label{thm:proximal:2} every proximal pair of $f$ is Banach proximal,
		\item\label{thm:proximal:3} the proximal relation $\mathrm{Prox}(f)$ is an equivalence relation.
	\end{enumerate}
\end{thm}

The last main result of this paper is the following theorem.
\begin{thm}\label{thm:LY=nIT}
	Let $f\colon G\to G$ be a graph map. The following conditions are equivalent:
	\begin{enumerate}
		\item\label{thm:LY=nIT:1} $f$ does not have scrambled pairs,
		\item\label{thm:LY=nIT:2} $f$ is tame,
		\item\label{thm:LY=nIT:3} $f$ is null.
	\end{enumerate}
\end{thm}

In the case of interval maps, this result is a consequence of \cite{FS91}
by Franzova and Sm\'\i tal (equivalence of (1) and (3))
and \cite{GY09} by Glasner and Ye (equivalence of (2) and (3)).
This was later generalized even further in \cite{L11} where it is proved that
in interval maps IN-pairs and IT-pairs coincide
(see subsection \ref{subsec:IN-IT-pairs} for the definitions).
Even though the results look similar, we adopt here a different approach compared to \cite{FS91}. It is also worth emphasizing that some tools used for interval maps cannot
be applied in the case of graph maps. Because of these difficulties, we are not able to mimic results of \cite{L11} in Theorem~\ref{thm:LY=nIT}, proving only that IN-pairs and IT-pairs coexist.
We leave the later statement as an open question.

\begin{que}
Let $f\colon G\to G$ be a graph map with zero topological entropy.
Is it true that: $(x,y)$ is an IN-pair if and only if $(x,y)$ is an IT-pair?
\end{que}

\section{Preliminaries}
Throughout this paper, let $\N$, $\N_0$, $\Z$ and $\mathbb{R}$
denote the set of all positive integers, non-negative integers, integers and real numbers, respectively.
The cardinality of a set $A$ is denoted by $|A|$.

An \emph{arc} is any topological space homeomorphic to the closed unit interval $[0,1]$.
A \emph{(topological) graph} is a non-degenerated compact
connected metric space $G$ which can be written as a union of finitely many
arcs, any two of which are disjoint or intersect only in their endpoints
(i.e., it is a non-degenerated connected compact one-dimensional polyhedron in $\mathbb{R}^3$).
By a \textit{graph map} we mean any continuous map $f\colon G\to G$ acting on a topological graph $G$.

By a \textit{topological dynamical system}, we mean a pair $(X, f)$, where $X$
is a compact metric space with a  metric $d$ and $f\colon X\to X$ is a continuous map.

A subset $K$ of $X$ is \emph{$f$-invariant} (or simply \emph{invariant}) if $f(K)\subset K$.
If $K$ is a closed $f$-invariant subset of $X$, then $(K,f|_K)$  is also a dynamical system.
We will call it a \emph{subsystem} of $(X,f)$.
If there is no ambiguity, for simplicity we will write $f$ instead of $f|_K$.
For a point $x\in X$, the \emph{orbit} of $x$, denoted by  $\Orb_f(x)$,
is the set $\{f^n(x): n\in\N_0\}$,
and the \emph{$\omega$-limit set} of $x$, denoted by $\omega_f(x)$,
is the set of limit points of the sequence $(f^n(x))_{n\in\N_0}$.

A point $x\in X$ is \emph{periodic} with the least period $n$ if
$n$ is the smallest positive integer satisfying $f^n(x) = x$,
and \emph{non-wandering} if for every open neighborhood $U$ of $x$
there exists $n\in\mathbb{N}$ such that $U\cap f^{-n}(U)\neq\emptyset$.
The set of non-wandering points is denoted by $\Omega(f)$.
The set $\omega(f)=\bigcup_{x\in X}\omega_f(x)$ is called the \emph{$\omega$-limit set} of $f$.
It is clear that $\omega(f)\subset \Omega(f)$.

A dynamical system $(X,f)$ is  \emph{minimal}
if it does not contain any non-empty proper subsystem.
A point $x\in X$ is \emph{minimal} if it belongs to some minimal subsystem of $(X,f)$.
By the Zorn's Lemma, it is not hard to see that every dynamical system has a minimal subsystem
and then there exists some minimal point.
A dynamical system $(X,f)$ is \emph{transitive} if $X=\omega_f(x)$ for some $x\in X$
and such a point $x$ is called a \textit{transitive point}.

We refer the reader to the textbook \cite{W82} for information on topological entropy and
topological sequence entropy.

\subsection{Mean equicontinuity and mean sensitivity}
Recall that a dynamical system $(X, f)$ is \emph{equicontinuous}
if for every $\eps>0$ there is $\delta>0$ with
the property that for every two points $x,y\in X$,
$d(x,y)<\delta$ implies $d(f^n(x),f^n(y))<\eps$ for all $n\in\N_0$.
Equicontinuous systems have simple dynamical behaviors.
It is well known that a dynamical system  $(X,f)$ with
$f$ being surjective is equicontinuous if and only if there exists a compatible metric $\rho$
on $X$ such that $f$ acts on $X$ as an isometry,
i.e., $\rho(f(x),f(y))=\rho(x,y)$ for all $x,y\in X$.

In \cite{F51} Fomin introduced the following weaker version of the equicontinuity condition.
A dynamical system $(X, f)$ is \emph{mean equicontinuous} or
\emph{mean Lyapunov stable}
 if for every $\eps>0$ there is $\delta>0$ with
the property that for every two points $x,y\in X$ condition
$d(x,y)<\delta$ implies
\[\limsup_{n\to\infty}\frac{1}{n}\sum_{k=0}^{n-1}d(f^k(x),f^k(y))<\eps.\]
It is shown in~\cite{LTY15} that every ergodic invariant measure of a mean equicontinuous system
has discrete spectrum and then the topological entropy of a mean equicontinuous system is zero.

A property opposite to mean equicontinuity is called mean sensitivity. Strictly speaking,
a dynamical system $(X, f)$ is \emph{mean sensitive} if there exists $\delta > 0$
such that for every $x \in X$ and every neighborhood $U$ of $x$, there exists $y \in U$ such that
\[\limsup_{n\to\infty}\frac{1}{n}\sum_{k=0}^{n-1}d(f^k(x),f^k(y))>\delta.\]
We have the following dichotomy result for minimal dynamical systems:
every minimal system is either mean equicontinuous or mean sensitive (see \cite{LTY15} or \cite{G15}).
While there exist transitive dynamical systems with positive entropy which are not mean sensitive (e.g. see \cite[Corollary~4.8]{LTY15}), the following result holds (see \cite{G15} or \cite{L16}).

\begin{thm}\label{thm:mean-sensitivity}
Let $(X,f)$ be a dynamical system.
If there exists an ergodic invariant measure with full support and positive entropy,
then $(X,f)$ is mean sensitive.
\end{thm}

\begin{cor} \label{cor:positive-entropy-subsystem}
Let $(X,f)$ be a dynamical system.
If topological entropy of $(X,f)$ is positive,
then there exists a mean sensitive transitive subsystem.
\end{cor}
\begin{proof}
By the variational principle of topological entropy,
there exists an ergodic invariant measure with positive entropy.
Applying Theorem~\ref{thm:mean-sensitivity} to $f$   restricted
to the support of this measure,
we obtain a mean sensitive transitive subsystem.
\end{proof}

\subsection{The structure of \texorpdfstring{$\omega$}{omega}-limit sets for graph maps}
In this subsection
let us recall a few results on the structure of $\omega$-limit sets for graph maps,
which were studied by Blokh in the series of papers \cite{Blokh82,Blokh,Blokh90a,Blokh90b}
and \cite{Blokh90c}.
Here we will follow the notations from~\cite{RS14}. 
Let $f\colon G\to G$ be a graph map. A subgraph $K$ of $G$
is called \emph{periodic} if
there exists a positive integer $k$ such that $K,f(K),\dotsc,f^{k-1}(K)$
are pairwise disjoint and $f^k(K)=K$.
In such a case, $k$ is called the \emph{period} of $K$, and
$K$ is a \emph{periodic (set) of period~$k$}.
The set $Orb_f(K)=\bigcup_{i=0}^k f^i(K)$ is called a \emph{cycle}
of graphs of period $k$.

For an infinite $\omega$-limit set $\omega_f(x)$, define
\begin{equation}
  \mathcal{C}(x):=\{X\subset G\colon X\text{ is a cycle of graphs and }\omega_f(x)\subset X\}.
\end{equation}
Note that $\bigcap_{n=1}^\infty f^n(G)$ is always a $1$-periodic cycle of graphs contained in $\mathcal{C}(x)$, hence always $\mathcal{C}(x)\neq \emptyset$.
Denote
$$
\mathcal{C}_P(x)=\sup\{ n\in \N : \text{ there is a subgraph }K \text{ with period }n\text{ and }Orb_f(K)\in \mathcal{C}(x)\}
$$
and observe that $\mathcal{C}_P(x)\in \N\cup \{+\infty\}$.

\begin{defn}
An infinite $\omega$-limit set $\omega_f(x)$ of a graph map is called a \emph{solenoid}
if $\mathcal{C}_P(x)=+\infty$.
\end{defn}

\begin{lem}[{\cite[Lemma~11]{RS14}}] \label{lem:solenoid}
If $\omega_f(x)$ is a solenoid, then
there exists a sequence of cycles of graphs $(X_n)_{n\in\N}$ with periods $(k_n)_{n\in \N}$ such that
\begin{enumerate}
  \item $(k_n)_{n\in \N}$ is strictly increasing, $k_1\geq 1$ and $k_{n+1}$ is a multiple of $k_n$ for all $n\geq 1$;
  \item for all $n\geq 1$, $X_{n+1}\subset X_n$ and every connected component of $X_n$
  contains the same number (equal to $k_{n+1}/k_n\geq 2$) of components of $X_{n+1}$;
  \item $\omega_f(x)\subset \bigcap_{n\geq 1}X_n$ and $\omega_f(x)$ does not contain periodic points.
\end{enumerate}
\end{lem}

\begin{defn}
Let $f\colon X\to X$ and $g\colon Y\to Y$ be two continuous maps.
A continuous map $\varphi\colon X\to Y$ is a \emph{semi-conjugacy} between $f$ and $g$
if $\varphi$ is onto and $\varphi\circ f=g\circ \varphi$.
If in addition $\varphi$ is a homeomorphism, then $\varphi$ is called a \emph{conjugacy} between $f$ and $g$.

Assume further that $K\subset X$ is a closed set such that $f(K)\subset K$.
We say that a semi-conjugacy $\varphi$ between $f$ and $g$ is an \emph{almost conjugacy}
between $f|_K$ and $g$ if it satisfies the following conditions:
\begin{enumerate}
  \item $\varphi(K)=Y$;
  \item for any $y\in Y$, $\varphi^{-1}(y)$ is a connected subset of $X$;
  \item for any $y\in Y$, $\varphi^{-1}(y)\cap K=\partial \varphi^{-1}(y)$,
  where $\partial$ denotes the boundary in $X$.
\end{enumerate}
\end{defn}

Observe that when $f$ is a graph map then in the above case, since $\varphi^{-1}(y)$ is connected, all sets $\partial \varphi^{-1}(y)$ are finite and their cardinality is uniformly bounded, that is
 there is $N>0$ such that $1\leq |\varphi^{-1}(y)\cap K|\leq N$ for every $y\in Y$.

\begin{defn}
Let $X$ be a finite union of subgraphs of $G$ such that $f(X)\subset X$. We define
\begin{equation}
  E(X,f)=\{y\in X\colon \text{ for any neighborhood } U \text{ of }y \text{ in } X,
  \overline{Orb_f(U)}=X\}.
\end{equation}
\end{defn}

The following result was first proved by Blokh in \cite{Blokh90a}.
Here we adopt the version which
combines statements of \cite[Lemmas~13 and 14]{RS14} and \cite[Proposition~16]{RS14}.

\begin{lem}\label{lem:non-solenoid}
If an infinite $\omega$-limit set $\omega_f(x)$ is not a solenoid,
then there exists a cycle of graphs $X\in\mathcal{C}(x)$ such that
\begin{enumerate}
  \item for any $Y\in \mathcal{C}(x)$, $X\subset Y$;
  \item the period of $X$ is maximal among the periods of all cycles in $\mathcal{C}(x)$.
\end{enumerate}
If we put $E=E(X,f)$ then:
\begin{enumerate}
  \item $\omega_f(x)\subset E\subset X$;
  \item $E$ is a perfect set and $f|_E$ is transitive;
  \item there exists a transitive map $g\colon Y\to Y$,
  where $Y$ is a finite union of graphs, and a semi-conjugacy $\varphi\colon X\to Y$
  between $f|_X$ and $g$ which almost conjugates $f|_E$ and $g$.
\end{enumerate}
\end{lem}

\begin{defn}
Assume that $\omega_f(x)$ is an infinite  set but not a solenoid.
Let $X$ be the minimal (in the sense of inclusion) cycle of graphs containing $\omega_f(x)$ and denote $E=E(X,f)$.
We say that $E$ is a \emph{basic set} if $X$ contains a periodic point, and \emph{circumferential set} otherwise.
\end{defn}

The next result is due to Blokh \cite{Blokh}.
Here we adopt the version which
combines statements of \cite[Theorems 20 and 22]{RS14}.
\begin{thm} \label{thm:basic-set-and-rotation}
Let $f\colon G\to G$ be a graph map.
\begin{enumerate}
  \item If $f$ admits a basic set, then the topological entropy of $f$ is positive.
  \item If $f$ is transitive map without periodic points then
  $f$ is conjugate to an irrational rotation of the circle.
\end{enumerate}
\end{thm}

\section{Local mean equicontinuity and M\"obius disjoint conjecture}
\subsection{Proof of Theorem~\ref{thm:main1}}
Inspired by the definition of local equicontinuity in \cite{GW00},
we will consider the following local version of mean equicontinuity.
\begin{defn}
A dynamical system $(X,f)$ is called \emph{locally mean equicontinuous}
if for every $x\in X$ the dynamical system $(\overline{\Orb_f(x)},f)$ is mean equicontinuous.
\end{defn}
By Corollary~\ref{cor:positive-entropy-subsystem},
we know that if a dynamical system is locally mean equicontinuous then it has zero topological entropy, therefore
to prove Theorem~\ref{thm:main1} we only need to prove the sufficiency.
To do this, we shall prove that $(\overline{\Orb_f(x)},f)$ is mean equicontinuous for all $x\in G$.
According to the structure of $\omega$-limit sets of graph maps,
we divide the proof into three cases: $\omega_f(x)$ is finite,
$\omega_f(x)$ is a solenoid and $\omega_f(x)$ is infinite but not a solenoid.

We start with the following simple observation covering the first of three possible cases mentioned above. We present the proof for the convenience of the reader.

\begin{lem}\label{lem:finite}
Let $(X,f)$ be a topological dynamical system and let $x\in X$.
If $\omega_f(x)$ is finite, then $(\overline{\Orb_f(x)},f)$ is equicontinuous.
\end{lem}
\begin{proof}
It is well known that if $\omega_f(x)$ is finite then it must be a finite periodic orbit
(see e.g. \cite[Lemma~1.4]{R15}).
There is a periodic point $y\in X$ with periodic $p$ such that $\lim_{n\to\infty}f^{pn+i}(x)=f^i(y)$
for $i=0,1,\dotsc,p-1$, which easily implies that $(\overline{\Orb_f(x)},f)$ is equicontinuous.
\end{proof}

The following result covers the second case.
\begin{lem}\label{lem:solenoid2}
Let $f\colon G\to G$ be a graph map and $x\in G$.
If $\omega_f(x)$ is a solenoid, then $(\overline{Orb(x,f)},f)$ is mean equicontinuous.
\end{lem}

\begin{proof}
	We follow the idea from the proof of \cite[Proposition 6.8]{LTY15}.
	Since $\omega_f(x)$ is a solenoid, by Lemma~\ref{lem:solenoid}
	there exists a sequence of cycles of graphs $(X_n)_{n\in \N}$ with periods $(k_n)_{n\in \N}$ such that
	$\lim_{n\to\infty} k_n=+\infty$.
	Denote the sum of lengths of all edges in $G$ by~$M$.
	Clearly, for every $\eps>0$ and any graphs $G_1,\ldots, G_k\subset G$ with disjoint interiors we have
	$|\{i : \diam (G_i)\geq \eps\}|\leq M/\eps$.
	
	For every $n\in\N$, sets $X_n,f(X_n),\dotsc,f^{k_n-1}(X_n)$ are pairwise disjoint closed subgraphs of $G$
	and for sufficiently large $n$ we have $k_n\geq \frac{2M}{\eps^2}$.
	Then for any $u,v\in X_n$, we have
	\begin{eqnarray*}
		\limsup_{N\to\infty}\frac{1}{N}\sum_{k=0}^{N-1}d(f^k(u),f^k(v))
		&\leq& \limsup_{N\to\infty}\frac{1}{N}\sum_{k=0}^{N-1}\diam(f^k(X_n))\\
		&\leq& \eps+\limsup_{N\to\infty}\frac{1}{N}|\{k<N : \diam(f^k(X_n))>\eps\}|\\
		&\leq& \eps+\limsup_{N\to\infty}\frac{1}{N} \left(\frac{N}{k_n}+1\right)\cdot\frac{M}{\eps}
		\;\leq\; \eps+\frac{2}{k_n}\cdot\frac{M}{\eps}\\
		&\leq& 2\eps.
	\end{eqnarray*}
	
Since $\omega_f(x)$ is infinite and the boundary of $\bigcup_{i=0}^{k_n-1}f^i(X_n)$ is a finite set,
	we must have $\Int(f^i(X_n))\cap \omega_f(x)\neq \emptyset$. Therefore
	there exists an integer $s\geq k_n$ such that $f^s(x)\in  X_n$.
	Let $r_n=\min_{i\neq j\; (\text{mod } k_n)}d(f^i(X_n),f^j(X_n))>0$.
	There exists a $\delta>0$ such that if $d(u,v)<\delta$ then $d(f^i(u),f^i(v))<\frac{r_n}{2}$ for $i=0,1,\ldots,s$.
This implies that if $u,v\in (\overline{Orb(x,f)},f)$ and $d(u,v)<\delta$,
then there exists an integer $m\in (0,s]$ such that $f^{m}(u),f^{m}(v)\in X_n$ and as a consequence
	\[\limsup_{N\to\infty}\frac{1}{N}\sum_{k=0}^{N-1}d(f^k(u),f^k(v))\leq 2\eps.\]
	This implies that $ (\overline{Orb(x,f)},f)$ is mean equicontinuous.
\end{proof}

Now let us consider the third case.
We need the following characterization of mean equicontinuity obtained in \cite{DG16}.
Let $(X, f)$ be an extension of $(Y, g)$ by a factor map $\pi$.
We say that $\pi$ is an \textit{isomorphic extension} if $(X, f)$ is uniquely ergodic
and $\pi$ is a measure-theoretic isomorphism between $(X, \mu, f)$ and $(Y, \nu, g)$.
It is shown in \cite[Theorem~2.1]{DG16} that
a minimal uniquely ergodic system is mean equicontinuous if and only if
it is an isomorphic extension of its maximal equicontinuous factor.

\begin{lem}\label{lem:circ}
	Let $f\colon G\to G$ be a graph map and $x\in G$.
	Assume that $\omega_f(x)$ is infinite  and not a solenoid.
	Let $X$ be the minimal cycle of graphs containing $\omega_f(x)$ and $E=E(X,f)$.
	If $E$ is circumferential, then $(X,f)$ is mean equicontinuous.
	In addition, $(\overline{Orb(x,f)},f)$ is mean equicontinuous.
\end{lem}

\begin{proof}
	By Lemma~\ref{lem:non-solenoid}
	there exists a transitive map $g\colon Y\to Y$,
	where $Y$ is a finite union of graphs, and a semi-conjugacy $\varphi\colon X\to Y$
	between $f|_X$ and $g$ which almost conjugates $f|_E$ and $g$.
	Without loss of generality, we assume that $Y$ is a graph instead of a finite union of graphs,
	using $f^k$ in place of $f$.
	By Theorem~\ref{thm:basic-set-and-rotation}, we may assume that $(Y,g)$ is a rotation of the unit circle.
For every $y\in Y$, if $\varphi^{-1}(y)$ is not a singleton, then $\varphi^{-1}(y)$
is a closed subgraph and therefore there are at most countably many such points $y$. 	Furthermore, sets $f^i(\varphi^{-1}(y))$ are pairwise disjoint closed subgraphs, so in particular
$\lim_{i\to\infty}\diam(f^i(\varphi^{-1}(y)))=0$.
Therefore, if we denote $D=\{y : |\varphi^{-1}(y)|>1\}$ then $\nu(D)=0$ for Haar measure of $(Y,g)$
and $\mu(\varphi^{-1}(D))=0$ for any invariant measure $(E,f)$.
This shows that $\varphi\colon E\to Y$ is a measure-theoretic isomorphism and so $(E,f)$ is uniquely ergodic.
Then $\varphi$ is an isomorphic extension, showing that $(X,f)$ is an isomorphic extension of its maximal equicontinuous factor. By \cite[Theorem~2.1]{DG16} we obtain that $(E,f)$ is mean equicontinuous.
	
Fix any $p\in X$ and $\eps>0$.
If $p\not\in E$, then there exists $y\in Y$ such that $p\in \Int(\varphi^{-1}(y))$ and then
	there also exists $\delta>0$ such that $B(p,\delta)\subset \Int(\varphi^{-1}(y))$.
But then for any $z\in X$ with $d(p,z)<\delta$ we have
	\[0=\limsup_{k\to\infty}d(f^k(p),f^k(z))=\limsup_{N\to\infty}\frac{1}{N}\sum_{k=0}^{N-1}d(f^k(p),f^k(z)).\]
	
Next assume that $p\in E$ and let $\delta>0$ be provided to $p$ and $\eps>0$ by mean equicontinuity of $f|_E$.
Fix any $z\in X$ with $d(p,z)<\delta$ and let $q\in \partial \varphi^{-1}(\varphi(z))\subset E$ with $d(p,q)\leq d(p,z)$.
Then
	\[\limsup_{N\to\infty}\frac{1}{N}\sum_{k=0}^{N-1}d(f^k(p),f^k(q))\leq \eps\]
and
	\[\limsup_{k\to\infty}d(f^k(z),f^k(q))=0,\]
which together give
	\[\limsup_{N\to\infty}\frac{1}{N}\sum_{k=0}^{N-1}d(f^k(p),f^k(z))\leq \eps.\]
This proves that $(X,f)$ is mean equicontinuous.
	
Since there exists $s\geq 0$ such that $f^s(x)\in X$, we easily see that $(\overline{Orb(x,f)},f)$ is mean equicontinuous. 
\end{proof}

\begin{proof}[Proof of Theorem~\ref{thm:main1}]
As we mentioned before, we only need  to prove the sufficiency, hence
to complete the proof, it is enough to combine Lemmas~\ref{lem:finite}, \ref{lem:solenoid2} and \ref{lem:circ}.
\end{proof}

\subsection{Sarnak's M\"obius disjointness conjecture for graph maps
with zero topological entropy}\label{sec:SMDC}

Recently, Fan and Jiang \cite{FJ15} introduced the notion of oscillating sequences as follows.
\begin{defn}
A sequence $(c_n)_{n\in\N}$ of complex numbers is called
an \emph{oscillating sequence} if for every fixed $t\in [0,1)$ it satisfies the oscillating condition
\[\lim_{N\to\infty}\frac{1}{N}\sum_{n=1}^N c_ne^{-2\pi in t}=0\]
and for some $\lambda>0$ and $M>0$ it satisfies the growth condition
\[\frac{1}{N}\sum_{n=1}^N |c_n|^\lambda \leq M\]
for all $N\geq 1$.
\end{defn}
There are many arithmetic functions, including the  M\"obius function
and the Liouville function, which are oscillating sequences (see \cite{DD82} and \cite{D94}).
We present the next definition after \cite{FJ15}.
\begin{defn}
We say that $(X,f)$ is \emph{minimally mean Lyapunov stable} if
every minimal subsystem of $(X,f)$ is mean equicontinuous.

We say that $(X,f)$ is \emph{minimally mean attractable}
if for any $x\in X$  there is a minimal subsystem $(K_x,f)$ such that
for any $\eps>0$ there exists $y\in K_x$ with
\[\limsup_{n\to\infty}\frac{1}{n}\sum_{k=0}^{n-1}d(f^k(x),f^k(y))<\eps.\]
\end{defn}

The following fact presents a simple relation between the introduced properties.

\begin{prop} \label{prop:LME-implies-MMA-MMLS}
If a dynamical system $(X,f)$ is locally mean equicontinuous,
then it is minimally mean attractable and minimally mean Lyapunov stable.
\end{prop}
\begin{proof}
Let $(K,f)$ be any minimal subsystem of $(X,f)$ and fix any
$x_0\in K$. Then $K=\overline{\Orb_f(x_0)}$ and $(K,f)$ is mean equicontinuous by the definition of
local mean equicontinuity. So $(X,f)$ is minimally mean Lyapunov stable.

Next, fix any $x\in X$ and $\eps>0$. As $(\overline{\Orb_f(x)},f)$ is mean equicontinuous,
there exists $\delta>0$ such that for any $u,v\in \overline{\Orb_f(x)}$,
$d(u,v)<\delta$ implies that
\[\limsup_{n\to\infty}\frac{1}{n}\sum_{k=0}^{n-1}d(f^k(u),f^k(v))<\eps.\]
Fix any minimal point $y_0\in \overline{\Orb_f(x)}$ and let
$m\in\N_0$ such that $d(f^m(x),y_0)<\delta$.
Obviously restriction of $f$ to $\overline{\Orb_f(y_0)}$ is surjective, hence
there exists a minimal point $y\in\overline{\Orb_f(y_0)}$ such that $f^m(y)=y_0$.
Thus
\[\limsup_{n\to\infty}\frac{{1}}{n}\sum_{k=0}^{n-1}d(f^k(x),f^k(y))=
\limsup_{n\to\infty}\frac{1}{n}\sum_{k=0}^{n-1}d(f^k(f^m(x)),f^k(y_0))<\eps. \]
This implies that $(X,f)$ is minimally mean attractable.
\end{proof}

In \cite{FJ15}, Fan and Jiang proved that any oscillating sequence
is linearly disjoint from all dynamical systems which are minimally mean attractable and
minimal mean Lyapunov stable.
They also provided several examples of minimally mean attractable and minimally mean Lyapunov stable
systems, including all $p$-adic polynomials, all
$p$-adic rational maps with good reduction, all automorphisms of
$2$-torus with zero topological entropy, all diagonalized affine maps
of $2$-torus with zero topological entropy, all Feigenbaum zero topological
entropy flows, and all orientation-preserving circle homeomorphisms.
Motivated by these results, Fan and Jiang conjecture in
\cite[Remark 8]{FJ15} the following.
\begin{conj}\label{conj:FJ}
All continuous interval maps with
zero topological entropy are minimally mean attractable and minimally mean Lyapunov stable.
\end{conj}

Theorem~\ref{thm:main2} provides a positive answer for that
conjecture\footnote{After this paper was accepted, the authors noticed
that Jiang also prove Conjecture~\ref{conj:FJ} independently in \cite{J17}.}.

\begin{proof}[Proof of Theorem~\ref{thm:main2}]
Combining Theorem~\ref{thm:main1} and Proposition~\ref{prop:LME-implies-MMA-MMLS}
we obtain that the Conjecture~\ref{conj:FJ} holds for all graph maps with zero topological entropy, which automatically proves Theorem~\ref{thm:main2}.
\end{proof}

As the last result of this section, we provide an example showing that
there exists a topological dynamical system $(X,f)$ which is
minimally mean attractable and minimally mean Lyapunov stable but not locally mean equicontinuous.
\begin{exam}
Let $I=\{0\}\times [0,1]\subset\mathbb{R}^2$ and
let $x_n=(1/n,a_n)$, for $n\geq 1$ where $a_n\in [0,1]$ is some sequence such that $|a_n-a_{n+1}|<1/n$.
Let $X=I\cup \{x_n\colon n\geq 1\}$ and define a map $f\colon X\to X$
by $f(x)=x$ for $x\in I$ and $f(x_n)=x_{n+1}$ for all $n\geq 1$.
It is clear that $X$ is compact and $f$ is continuous. Then we obtain a dynamical system $(X,f)$.
It is easy to see that $(X,f)$ is minimally mean Lyapunov stable since the only minimal
subsets are singletons on $I$.

Now let us consider a particular example of $(X,f)$, by putting
$$
a_n=\begin{cases}
0, &n<2^{10} \textrm{ or } n\in [2^k,2^{k+1}-2k) \textrm{ for some }k\geq 10,\\
i/k, & n=2^{k+1}-2k+i \textrm{ for some }k\text{ and } 0\leq i<k,\\
1-i/k,& n=2^{k+1}-k+i \textrm{ for some }k\text{ and } 0\leq i<k.
\end{cases}
$$
Clearly for any $j\in [2^k,2^{k+1})$ we have $|\{ n<j: a_n\neq 0  \}|\leq \sum_{i=1}^{k+1}2i\leq 2(k+1)^2$
hence for every $x\in \{x_n\colon n\geq 1\}$ we have
$$
\limsup_{n\to\infty}\frac{1}{n}\sum_{i=0}^{n-1}d(f^i(x),f^i(0,0))
\leq \lim_{k\to\infty}\frac{2(k+1)^2}{2^k}=0
$$
which shows that $(X,f)$ is minimally mean attractable.
But it is not locally mean equicontinuous, because $I\subset \omega_f(x_1)$.
\end{exam}

\section{Li-Yorke chaos and proximality for graph maps with
zero topological entropy}\label{sec:scrambled:entropy}
The definition of scrambled set as a determinant of chaos originated
from the paper of Li and Yorke \cite{LY}.
During last 40 years it generated a lot of attention and motivated many new
results (see \cite{BHS08} and \cite{LY16} for presentation of recent advances on that field).
The following definition is a more modern statement of ideas from \cite{LY}.
\begin{defn}
Let $(X,f)$ be a dynamical system.
A pair $(x,y)\in X^2$ is called: \emph{proximal} if $\liminf_{n\to\infty}d(f^n(x),f^n(y))=0$;
\emph{asymptotic} if $\lim_{n\to\infty}d(f^n(x),f^n(y))=0$;
\emph{scrambled} or \emph{Li-Yorke} if it is proximal but not asymptotic.

We say that a dynamical system $(X,f)$ is \textit{chaotic in the sense of Li and Yorke}
if there exists an uncountable subset $S$ of $X$ such that every two distinct points in $S$
form a scrambled pair.
\end{defn}

In \cite{KS89}, Kuchta and Simital proved that an interval map has a scrambled pair
if and only if it is chaotic in the sense of Li and Yorke.
This result was later generalized to graph maps in \cite{RS14}.
The following lemma is implicitly contained in the proof of \cite[Theorem~3]{RS14}.
We present a complete proof for the reader convenience.

\begin{lem}\label{lem:scrambled-soleniod}
Let $f\colon G\to G$ be a graph map with zero entropy.
If $(x,y)$ is a scrambled pair,
then both $\omega_f(x)$ and $\omega_f(y)$ are solenoids.
\end{lem}
\begin{proof}
It is clear that $\omega_f(x)$ and $\omega_f(y)$ can not be both finite.
First, we claim that if $\omega_f(x)$ is infinite then $\omega_f(y)$ must be infinite as well.
If $\omega_f(y)$ is finite, it must be a periodic orbit.
As $(x,y)$ is a proximal pair, $\omega_f(x)\cap \omega_f(y)\neq\emptyset$.
This implies that $\omega_f(x)$ properly contains a periodic point,
hence $\omega_f(x)$ is a basic set (it is not solenoid by Lemma~\ref{lem:solenoid}).
Then entropy of $f$ is positive by Theorem~\ref{thm:basic-set-and-rotation} which is a contradiction.
The claim is proved.

Furthermore, by Theorem~\ref{thm:basic-set-and-rotation}, to show that $\omega_f(x)$ is a solenoid
we only need to  show that  $\omega_f(x)$ cannot be circumferential.
Assume on the contrary that $\omega_f(x)$ is circumferential and
let $K_1, \ldots, K_n$ be the minimal cycle of graphs containing $\omega_f(x)$ without periodic points.
Then putting $K=K_i$ for well chosen $i$,
we obtain that $f^n(K)\subset K$ and $\omega_{f^n}(x)\subset K$,
because $\omega_f(x)=\bigcup_{i=0}^{n-1}\omega_{f^n}(f^i(x))$.
Since $x,y$ are proximal, $K$ is a graph and $K_i\cap K_j=\emptyset$ for $i\neq j$, without loss of generality we may assume that $x,y\in K$.
Denote $g=f^n|_K$. Observe that there does not exist cycle of graphs refining $K$ and
there is also no periodic point of $g$ in $K$ (that is, $K$ is circumferential for $g$).
By Lemma~\ref{lem:non-solenoid}, if we put $E=E(X,g)$ then
$\omega_g(x)\subset E\subset K$, $E$ is a perfect set,
$g|_E$ is transitive and there exists a transitive map $h\colon Y\to Y$,
where $Y$ is graph, and a semi-conjugacy $\varphi\colon X\to Y$
between $g|_K$ and $h$ which almost conjugates $g|_E$ and $h$.
By the definition, every minimal set of $g$ in $K$ is infinite and hence also every minimal set of $h$
 is infinite (by definition $E\cap \varphi^{-1}(z)$ is finite for every $z\in Y$).
This shows that $h$ is a transitive graph map without periodic points,
which by result of Blokh \cite{Blokh} implies that $Y$ is a circle and
$h$ is (conjugated to) an irrational rotation.
Since proximal pairs are preserved by semi-conjugacy and $h$ does not have proximal pairs,
there exists $z$ such that $x,y\in \varphi^{-1}(z)$. But sets $\varphi^{-1}(h^i(z)), \varphi^{-1}(h^j(z))$
are pairwise disjoint subgraphs for $i \neq j$ and so $\lim_{j\to\infty}\diam \varphi^{-1}(h^j(z))=0$.
This is a contradiction, since $(x,y)$ is not an asymptotic pair. Indeed $\omega_f(x)$
cannot be circumferential.

To finish the proof note that $\omega_f(y)$ is also a solenoid by symmetry of the argument.
\end{proof}

\begin{defn}
We say that a dynamical system $(X,f)$ is \emph{Devaney chaotic} if $X$ is infinite,
the set of periodic points is dense in $X$ and $(X,f)$ is transitive.
\end{defn}
It is proved in \cite{L93},
that  an interval map has positive entropy if and only if
it has a Devaney chaotic subsystem.
By recent advances, this result can easily be generalized to graph maps, as presented below. It will be used as a mid-step in
the proof of Theorem~\ref{thm:scrambled3tuple}.

\begin{prop}\label{prop:Devaney-chaos}
Let $f\colon G\to G$ be a graph map. Then $(G,f)$ has positive entropy if and only if
it has a Devaney chaotic subsystem.
\end{prop}
\begin{proof}
If $(G,f)$ has a Devaney chaotic subsystem $(Y,f)$,
then $Y$ is an infinite $\omega$-limit set containing a periodic point,
hence $Y$ is a basic set.
By Theorem~\ref{thm:basic-set-and-rotation}, $f$ has positive entropy.

Now assume that $f$ has positive entropy.
By Theorem B and Lemma 3.3 in \cite{LM}, there exists $n\in\mathbb{N}$ such that
$f^n$ has a strong $2$-horseshoe, that is there exists an closed interval $I\subset G$ and
two disjoint closed intervals $J_1$ and $J_2$ such that $J_1\cup J_2\subset \Int(I)$
and $f^n(J_1)=f^n (J_2)=I$.
Repeating the argument from \cite[Proposition 4.7]{L93} (or \cite[Proposition 5.15]{R15}) we
find a closed $f^n$-invariant set $D\subset I$ such that $f^n|_D$ is almost
conjugate to the full shift on $2$-symbols.
By the proof of \cite[Proposition 4.9]{L93}, there exists a closed subset $E$ of $D$
such that $f^n|_E$ is Devaney chaotic.
Taking $F=\bigcup_{i=0}^{n-1}f^i(E)$ it is easy to see that $f|_F$ is Devaney chaotic.
\end{proof}

The following definition was introduced in \cite{X05} as a more sensitive tool for the description of
complexity degrees in dynamical systems which are chaotic in the sense of Li and Yorke.

\begin{defn}
Let $(X,f)$ be a topological dynamical system and $n\geq 2$.
A tuple $(x_1,\dotsc,x_n)\in X^n$ is called \emph{$n$-scrambled} if
\[
    \liminf_{k\to\infty}\max_{1\leq i<j\leq n} d\bigl(f^k(x_i),f^k(x_j)\bigr)=0,\quad
    \limsup_{k\to\infty}\min_{1\leq i<j\leq n}d\bigl(f^k(x_i),f^k(x_j)\bigr)>0.
\]
The dynamical system $(X,f)$ is called \emph{$n$-chaotic in the sense of Li-Yorke} if
there exists an uncountable subset $S$ of $X$ such that every $n$ pairwise distinct points in $S$
form an $n$-scrambled tuple.
\end{defn}

\begin{proof}[Proof of Theorem~\ref{thm:scrambled3tuple}]
If $f$ has positive entropy, by Proposition~\ref{prop:Devaney-chaos}
there exists a Devaney chaotic subsystem.
By the main result of \cite{X05}, Devaney chaos implies
$n$-chaos in the sense of Li-Yorke for all $n\geq 2$.
In particular, there is a scrambled $3$-tuple in $G$.

It remains to show that if $f$ has zero topological entropy,
then it has no $3$-scrambled tuples.
Assume that there exists a $3$-scrambled tuple $(x_1,x_2,x_3)$.
By Lemma~\ref{lem:scrambled-soleniod}, $\omega_f(x_1)$ is a solenoid and then
by Lemma~\ref{lem:solenoid} there exists a sequence of
cycles of graphs $X_n$ with strictly increasing periods such that
$\omega_f(x)\subset \bigcap_{n\geq 1} X_n$.

Since there are finitely many branching points in $G$,
replacing $f$ by $f^n$, we can find an index $i$ such that $I=X_i\subset G$ is an interval and $k\geq 0$ such that
we can view $\omega_{f^n}(f^k(x_1))$ as an invariant  subset of the interval map $f^n|_I$.
As $(x_1,x_2,x_3)$ is scrambled tuple and $\omega_{f^n}(f^k(x_1))\subset \omega_f(x_1)\cap I$ is an infinite minimal set,
there exists $j\geq 0$ such that $f^j(x_i)\in I$ for $i=1,2,3$.
Therefore $(f^j(x_1),f^j(x_2),f^j(x_3))$ is a scrambled tuple for the interval map $f^n|_I$,
which is impossible because $f^n|_I$ has zero topological entropy (see \cite{L11}).
\end{proof}

The proximal relation of a dynamical system $(X,f)$ is denoted by
\[
    \mathrm{Prox}(f)=\{(x,y)\in X^2\colon (x,y)\textrm{ is proximal}\}.
\]
The reflexivity and symmetry of $\mathrm{Prox}(f)$ is obvious.

A set $F\subset \N$ has \textit{Banach density one} if for every $\lambda<1$
there exists $N\geq 1$ such that $|F\cap A|\geq \lambda |A|$
for every block of consecutive integers $A\subset \N$ with  $|A|\geq N$.
We say that a pair $(x,y)\in X^2$ is \textit{Banach proximal} if for every
$\eps>0$, the set $\{n : d(f^n(x),f^n(y))<\eps\}$ has Banach density one.
It is clear that every asymptotic pair is Banach proximal and
every Banach proximal pair is proximal.

\begin{proof}[Proof of Theorem~\ref{thm:proximal}]

\eqref{thm:proximal:1} $\Longrightarrow$ \eqref{thm:proximal:2}:
Let $(x,y)$ be a proximal pair.
If $(x,y)$ is asymptotic, then it is clearly also Banach proximal, so we may assume that $(x,y)$ is not asymptotic.
Then $(x,y)$ is a scrambled pair and
by Lemma~\ref{lem:scrambled-soleniod}, $\omega_f(x)$ is a solenoid.
Proceeding as in the proof of Theorem~\ref{thm:scrambled3tuple},
we find integers $n,k>0$ and an interval $I\subset G$
such that $\omega_{f^n}(f^k(x))\subset I$ and $f^n(I)\subset I$.
Note that $(x,y)$ is proximal, hence
there exists $j\geq 0$ such that $f^j(x),f^j(y)\in I$.
So $(f^j(x),f^j(y))$ is a scrambled pair for the interval map $f^n|_I$,
By \cite[Proposition 4.3]{L11} we obtain that
$(f^j(x),f^j(y))$ is a Banach proximal pair for $f^n|_I$ and then obviously
$(x,y)$ is a Banach proximal pair for $f$.

\eqref{thm:proximal:2} $\Longrightarrow$ \eqref{thm:proximal:3}:
The intersection of two Banach density one sets also has Banach density one.
As every proximal pair is Banach proximal, it is easy to see that
 $\mathrm{Prox}(f)$ is an equivalence relation.

\eqref{thm:proximal:3} $\Longrightarrow$ \eqref{thm:proximal:1}:
If $f$ has positive topological entropy, by Proposition~\ref{prop:Devaney-chaos}
there exists a Devaney chaotic subsystem.
By \cite[Lemma~5.4]{L11}, $\mathrm{Prox}(f)$ is not an equivalence relation.
\end{proof}

\section{Null systems and chaos in the sense of Li and Yorke} \label{sec:null}
Let $(X,f)$ be a topological dynamical system.
For an increasing sequence $A$ of positive integers,
the sequence entropy of $(X,f)$ along the sequence $A$ is denoted by $h^A(T)$.
Let $h^*(T)=\sup_A h^A(X,f)$, where $A$ ranges over all increasing sequences of positive integers.
We say that $(X,f)$ is \emph{null} if $h^*(T)=0$.
In~\cite{FS91}, Franzova and Sm\'\i tal showed that an interval map
has no scrambled pairs if and only if it is null.

The enveloping semigroup $\mathcal{E}(X,f)$ of a dynamical system $(X,f)$
is defined as the closure of $\{f,f^2,\dotsc\}$ in the compact space $X^X$.
A dynamical system $(X,f)$ is \emph{tame} if
the cardinal number of its enveloping semigroup is not greater than the cardinal number of $\mathbb{R}$ \cite{Gla}.
If a dynamical system is null, then it is also tame.
In \cite{GY09} Glasner and Ye proved that an interval map
it is null if and only if it is tame. The aim of this section to prove a generalization of the above result, as stated in Theorem~\ref{thm:LY=nIT}.

\subsection{IN-pairs and IT-pairs}\label{subsec:IN-IT-pairs}
In this subsection, we recall some definitions and result on local entropy theory,
which will be used later.

\begin{defn}
Let $(X,f)$ be a dynamical system.
For a tuple $(A_1,\dotsc,A_k)$ of subsets of $X$,
we say that a subset $J\subset \mathbb{Z}_+$ is an \emph{independence set} for $(A_1,\dotsc,A_k)$
if for any non-empty finite subset $I\subset J$, we have
\[\bigcap_{i\in I} f^{-i}A_{s(i)}\neq\emptyset\]
for all $s\in\{1,\dotsc,k\}^I$.
\end{defn}
It is clear that if $J$ is an independence set for $(A_1,\dotsc,A_k)$
then for every $m\in\mathbb{N}$, $J+m=\{i+m\colon i\in J\}$ is also an
independence set for $(A_1,\dotsc,A_k)$.

\begin{defn}
We call a tuple $(x_1,\dotsc,x_k)\in X^k$
\begin{enumerate}
\item an \emph{IE-tuple} if for every product neighborhood $U_1\times \dotsb\times U_k$
of $(x_1,\dotsc,x_k)$, the tuple  $(U_1, \dotsc, U_k)$ has an independence set of positive density;
\item an \emph{IT-tuple} if  for every product neighborhood $U_1\times \dotsb\times U_k$
of $(x_1,\dotsc,x_k)$, the tuple  $(U_1, \dotsc, U_k)$ has an infinite independence set;
\item an \emph{IN-tuple} if for every product neighborhood $U_1\times \dotsb\times U_k$
of $(x_1,\dotsc,x_k)$,
the tuple  $(U_1, \dotsc, U_k)$ has arbitrarily long finite independence sets.
\end{enumerate}
\end{defn}

Denote by $IE(X,f)$, $IN(X,f)$ and $IT(X,f)$ the collection of IE-pairs,
IN-pairs and IT-pairs for $(X,f)$ respectively.
It is clear that $IT(X,f)\subset IN(X,f)\subset IE(X,f)$.

\begin{thm}[{\cite[Theorem 2.6]{TYZ10}}]\label{thm:IN-IT-pairs}
Let $(X,f)$ be a dynamical system.
\begin{enumerate}
\item If $(x,y)$ is an IN-pair, then $\{x,y\}\subset\Omega(f)$;
\item both $IT(X,f)$ and $IN(X,f)$ are closed and $f\times f$-invariant subsets of $X\times X$;
\item for any $k\in\mathbb{N}$,  $IT(X,f)=IT(X,f^k)$ and $IN(X,f)=IN(X,f^k)$.
\end{enumerate}
\end{thm}

\begin{thm}[{\cite[Theorem 2.9]{TYZ10}}] \label{thm:IN-pair-preimage}
Let $(X,f)$ be a dynamical system.
If $(x,y)$ is an IN-pair, then there exists an IN-pair $(x',y')$
such that $f(x')=x$ and $f(y')=y$.
\end{thm}

\begin{thm}
Let $(X,f)$ be a dynamical system. Then
\begin{enumerate}
\item\label{ind:1} $(X,f)$ has zero topological entropy if and only if every IE-pair is diagonal;
\item\label{ind:2} $(X,f)$ is tame if and only if every IT-pair is diagonal;
\item\label{ind:3} $(X,f)$ is null if and only if every IN-pair is diagonal.
\end{enumerate}
\end{thm}
We remark here that in the above theorem
\eqref{ind:1} was first proved by Huang and Ye \cite{HY06} using the notion of interpolating sets,
the necessary condition in \eqref{ind:2} was proved by Huang in \cite{H06} using the notion of scrambled pair.
A complete proof of \eqref{ind:1}--\eqref{ind:3} conditions was obtained by Kerr and Li \cite{KL07}, where they introduced the notion of independence.

\subsection{Proof of Theorem~\ref{thm:LY=nIT}}
First we need the following folklore result, which is an immediate consequence of Auslander-Ellis theorem (see e.g. \cite[Theorem 6.13]{A88}).
\begin{lem}\label{lem:min}
Let $(X,f)$ be a dynamical system.
If $f$ does not have Li-Yorke pairs, then for every $x\in X$, the set $\omega_f(x)$ is minimal.
\end{lem}

\begin{lem}\label{lem:omegas}
Let $f\colon G\to G$ be a graph map and assume that $f$ has no scrambled pairs.
If $(x,y)$ is a non-diagonal IN-pair and $\omega_f(x)$ is infinite, then $\omega_f(x)=\omega_f(y)$.
\end{lem}
\begin{proof}	
Let $K$ be a cycle of graphs containing $\omega_f(x)$.
Clearly, there is an $n$ such that $f^n(x)\in U \subset K$ where $U$ is a free arc (i.e. $U$ does not contain branching points).
Then there is $\delta>0$ such that $f^n(B_\delta(x))\subset U$ and so for every $z\in B_\delta(x)$ and
every $m>n$ we have $f^m(z)\in K$. This by definition of IN-pair and the fact that
$K$ is closed immediately implies that $x\in K$ and by the same argument $y\in K$.

If $\omega_f(x)$ is a solenoid then by Lemma~\ref{lem:solenoid} there exists a sequence
of cycles of graphs $X_n$ with strictly increasing periods such that
$\omega_f(x)\subset \bigcap_{n\geq 1} X_n$. Then also $\omega_f(y)\subset \bigcap_{n\geq 1} X_n$.
Since for every $\eps>0$ there is $n$ and a graph in the cycle $X_n$
with diameter smaller than $\eps$, we obtain that $\omega_f(x)\cap \omega_f(y)\neq \emptyset$.
By Lemma~\ref{lem:min} this implies $\omega_f(x)=\omega_f(y)$.

Next, suppose that $\omega_f(x)$ is not a solenoid. Since the topological entropy of $f$ is zero,
by Lemma~\ref{lem:non-solenoid} $\omega_f(x)$ is contained in a circumferential set.
By \cite[Lemma~14]{RS14} both $\omega_f(x)$, $\omega_f(y)$ are contained in the same maximal
$\omega$-limit set, which is a minimal set by Lemma~\ref{lem:min}. Again $\omega_f(x)=\omega_f(y)$.
\end{proof}

\begin{lem}\label{lem:no-LY-pairs1}
Let $f\colon G\to G$ be a graph map and assume that $f$ has no scrambled pairs.
For every $x\in G$, if $\omega_f(x)$ is infinite
then $(x,y)$ is not an IN-pair for every $y\in G\setminus\{x\}$.
\end{lem}
\begin{proof}
We prove this result by a contradiction.
Assume that there is $y\in G\setminus\{x\}$ such that $(x,y)$ is an IN-pair.
By Lemma~\ref{lem:omegas} we have $\omega_f(x)=\omega_f(y)$.

We have two cases on $\omega_f(x)$ to consider.
First assume that $\omega_f(x)$ is a solenoid.
By Lemma~\ref{lem:solenoid} there exists a sequence of
of cycles of graphs $X_n$ with strictly increasing periods such that $\omega(x)\subset \bigcap_{n\geq 1} X_n$.
There is $s>0$ such that $f^s(x)\in \Int(X_1)$. But by
the definition of IN-pair, for any $\delta>0$
there is $j>s$ and $u,v\in G$ such that $f^s(u),f^s(v)\in X_1$ and $d(f^j(u),x)<\delta$ and $d(f^j(u),y)<\delta$. This implies that $x,y\in X_1$ and by the same argument $x,y\in \bigcap_{n\geq 1} X_n$.
Since there are finitely many branching points, there are $k,m\geq 0$ and an interval $I$ such that
we can view $\omega_{f^m}(f^k(x))$ as an invariant subset of the interval map $f^m|_I$.
But $f^m|_I$ does not have scrambled pairs, hence for any $\delta>0$ and any sufficiently long cycle of graphs for $f^m|_I$ containing $\omega_{f^m}(f^k(x))$, their connected components
have diameters bounded by $\delta$ (e.g. see \cite{R15}). But if $Y$ is a cycle of graphs covering $\omega_{f^m}(f^k(x))$ then $\bigcup_{j=0}^{m-i}f^j(Y)$ is a cycle of graphs covering $\omega_f(x)$.
This implies that for any $\eps>0$ diameters of all connected components of $X_n$ are uniformly bounded by $\eps>0$ provided that $n=n(\eps)$ is sufficiently large. But since $x,y$ have to belong to the same connected component of $X_n$  for each $n$, we obtain that $x=y$ which is a contradiction.

Now assume that $\omega_f(x)$ is not a solenoid.
Let $K$ be a cycle of graphs containing $\omega_f(x)$. It is clear that $x,y\in K$
and since they are IN-pair, they have to be contained in the same connected component of $K$.
By Theorem~\ref{thm:IN-IT-pairs}, IN-pair for $f$ is IN-pair for $f^k$ for any $k$,
hence we may assume that $K$ is connected.
As $f$ has zero entropy, by Theorems~\ref{lem:non-solenoid} and \ref{thm:basic-set-and-rotation}
there exists an almost conjugacy $\pi \colon K \to Y$ between $f|_K$
and an irrational rotation on the unit circle $(Y,g)$.
This immediately implies that $\pi(x)=\pi(y)$.
Let $U,V$ be small intervals in $Y$ with $U\cap V=\pi(x)$. Since $\pi^{-1}(\pi(x))$ is a non-degenerate connected set,
there are pairwise disjoint open sets $U',V'\subset K$, $x\in U'$, $y\in V'$ such that $\pi(U')\subset U$ and $\pi(V')\subset V$.
Clearly there are no $p_1\in U'$,$p_2\in V'$ and $0<k<l$
such that $\pi(p_1),\pi (f^k(p_2)),\pi(f^l(p_1))\in U$ and $\pi(p_2),\pi (f^k(p_1)),\pi (f^l(p_2))\in V$
because rotation preserves local ordering of $Y$ which is a contradiction.
\end{proof}

\begin{lem}\label{lem:INinOmega}
Let $f\colon G\to G$ be a graph map.
If $(x,y)$ is an IN-pair then $\{x,y\}\subset \omega(f)$.
\end{lem}
\begin{proof}
By Theorem \ref{thm:IN-pair-preimage}, for every $n\in\mathbb{N}$,
there exists an IN-pair $(x_n,y_n)$ such that $f^n(x_n)=x$ and $f^n(y_n)=y$.
By Theorem \ref{thm:IN-IT-pairs}, we know that every point in an IN-pair is non-wandering.
So $\{x,y\}\subset \bigcap_{n=0}^\infty f^n(\Omega(f))$.
Now the result follows from the fact $\omega(f)=\bigcap_{n=0}^\infty f^n(\Omega(f))$
which is proved in \cite[Theorem 3]{MS07v2}.
\end{proof}

\begin{lem}\label{lem:no-LY-pairs2}
Let $f\colon G\to G$ be a graph map and assume that $f$ has no scrambled pairs.
For every $x\in G$, if $\omega_f(x)$ is finite
then $(x,y)$ is not an IN-pair for every $y\in G\setminus\{x\}$.
\end{lem}
\begin{proof}
We also prove this result by a contradiction.
Assume that there is $y\in G\setminus\{x\}$ such that $(x,y)$ is an IN-pair.
By Lemma~\ref{lem:omegas}, $\omega_f(y)$ is also finite.
Since $(x,y)$ is an IN-pair, by Lemma~\ref{lem:INinOmega}, $\{x,y\}\subset\omega(f)$.
Moreover,  By Lemma~\ref{lem:min}, every point in $\omega(f)$ is minimal.
As $\omega_f(x)$ and $\omega_f(y)$ are finite, $x$ and $y$ must be periodic points.
By Theorem~\ref{thm:IN-IT-pairs}, IN-pair for $f$ is also an IN-pair for $f^k$ for any $k$,
hence we may assume that both $x$ and $y$ are fixed points.
By \cite[Theorem 5.1]{T11}, $f$ has positive entropy. This is a contradiction.
\end{proof}

\begin{lem}\label{lem:tameLY}
Let $f\colon G\to G$ be a graph map.
If $f$ is  tame then there is no scrambled pair.
\end{lem}
\begin{proof}
The proof follows the same lines as the proof of \cite[Theorem~3]{RS14}.
It is enough to show that when there is a scrambled pair then there is also a non-diagonal IT-pair.
The statement is obvious when $h(f)>0$, so let us assume that $h(f)=0$.
Consider any scrambled pair $(x,y)$ in $G$.
By Lemma~\ref{lem:scrambled-soleniod}, $\omega_f(x)$ is a solenoid.
Since graphs forming cycle are pairwise disjoint and there are only finitely many branching points in $G$, by Lemma~\ref{lem:solenoid}
there exists an arc $I$ contained completely in an edge of $G$ and $n>0$ such that $f^n(I)\subset I$ and there is also an arc $J\subset \Int I$
such that $J \cap \omega_f(x)\neq \emptyset$. Since $x,y$ are proximal and there are $k,s>0$ such that $f^{ik+s}(x)\in J$ for every $i\geq 0$,
without loss of generality we may assume that $x,y\in I$. But then $f^n|_I\colon I\to I$ is a (possibly not surjective) interval map containing a scrambled pair.
By results of \cite{L11} there is a non-diagonal IT-pair of $f^n$ in $I$, which implies that $f^n$ has a non-diagonal IT-pair in $G$.
But IT-pair for $f^n$ is also an IT-pair for $f$ which completes the proof.
\end{proof}

\begin{proof}[Proof of Theorem~\ref{thm:LY=nIT}]
Implication $\eqref{thm:LY=nIT:1}\Longrightarrow\eqref{thm:LY=nIT:2}$ is a direct consequence of Lemma~\ref{lem:tameLY} and
$\eqref{thm:LY=nIT:2}\Longrightarrow\eqref{thm:LY=nIT:3}$ follows by definition. Combining Lemmas \ref{lem:no-LY-pairs1} and \ref{lem:no-LY-pairs2} we obtain the last implication
$\eqref{thm:LY=nIT:3}\Longrightarrow\eqref{thm:LY=nIT:1}$.
\end{proof}

\section*{Acknowledgements}
Research of J. Li was supported in part by NSF of China (grant numbers 11401362 and 11771246).
Research of P. Oprocha was supported by National Science Centre, Poland (NCN), grant no.
2015/17/B/ST1/01259.

Some part of results contained in this paper was obtained during research stay of J. Li and P. Oprocha at Fudan University in Feb 2017. The authors are grateful to Guohua Zhang for his support and hospitality during their visit in Shanghai.
The authors would like to thank the anonymous referees
for the careful reading and helpful suggestions.


\begin{thebibliography}{99}
\bibitem{A88} J. Auslander, \textit{Minimal flows and their extensions},
North-Holland Mathematics Studies, 153, North-Holland Publishing Co., Amsterdam, 1988.

\bibitem{BHS08} F. Blanchard, W.~Huang, L.~Snoha,
\textit{Topological size of scrambled sets},
Colloq. Math. \textbf{110} (2008), no. 2, 293--361.

\bibitem{Blokh82}
A. M. Blokh,  \textit{``Spectral expansion'' for piecewise monotone mappings of an interval} (Russian),
Uspekhi Mat. Nauk 37 (1982), no. 3(225), 175--176.

\bibitem{Blokh} A. M. Blokh, \textit{On transitive mappings of one-dimensional branched manifolds}
 (Russian), Differential-difference equations and problems of mathematical physics (Russian), 3--9,
 131, Akad. Nauk Ukrain. SSR, Inst. Mat., Kiev, 1984.

\bibitem{Blokh90a} A. M. Blokh,
\textit{Dynamical systems on one-dimensional branched manifolds I} (Russian),
Teor. Funktsi{\v\i} Funktsional. Anal. i Prilozhen. No. 46 (1986), 8--18; translation in
J. Soviet Math. 48 (1990), no. 5, 500--508.

\bibitem{Blokh90b} A. M.  Blokh,
\textit{Dynamical systems on one-dimensional branched manifolds II} (Russian),
Teor. Funktsi{\v\i} Funktsional. Anal. i Prilozhen. No. 47 (1987), 67--77; translation in
J. Soviet Math. 48 (1990), no. 6, 668--674.

\bibitem{Blokh90c} A. M.  Blokh,
\textit{Dynamical systems on one-dimensional branched manifolds III} (Russian),
Teor. Funktsi{\v\i} Funktsional. Anal. i Prilozhen. No. 48 (1987), 32--46; translation in
J. Soviet Math. 49 (1990), no. 2, 875--883.


\bibitem{DD82} H. Daboussi and H. Delange, \textit{On multiplicative arithmetrical functions
whose modulues does not exceed one}, J. London Math. Soc. (2) \textbf{26} (1982), no. 2, 245--264.

\bibitem{Dav}  H.  Davenport, \textit{On some infinite series involving arithmetical functions (II)},
 Quart. J. Math.  \textbf{8} (1937), 313-320.

\bibitem{D94} H. Delange, \textit{Sur les functions multiplicatives complexes de module},
Ann. Inst. Fourier, \textbf{44} (1994), no. 5, 1323--1349.

\bibitem{DG16} T. Downarowicz and E. Glasner, \textit{Isomorphic extensions and applications,}
Topol. Methods Nonlinear Anal. \textbf{48} (2016), 321--338.

\bibitem{FS91} N. Franzova and J. Sm\'\i tal,
\textit{Positive sequence topological entropy characterizes chaotic maps,} Proc. Amer. Math. Soc. \textbf{112} (1991), 1083--1086.

\bibitem{FJ15} A. Fan and Y. Jiang, \textit{Oscillating Sequences, MMA and MMLS Flows and Sarnak's Conjecture},
  Ergod. Th. Dynam. Sys., published online,
  DOI: 10.1017/etds.2016.121.

\bibitem{F51} S. Fomin, \textit{On dynamical systems with a purely point spectrum,} Dokl. Akad.
Nauk SSSR, \textbf{77} (1951), 29-32 (In Russian).

\bibitem{Gla} E. Glasner, \textit{On tame dynamical systems}, Colloq. Math., \textbf{105} (2006), 283--295.

\bibitem{GW00} E. Glasner and W. Weiss, \textit{Locally equicontinuous dynamical systems},
Colloq. Math. \textbf{84/85} (2000), part 2, 345--361.

\bibitem{GY09} E. Glasner and X. Ye, \textit{Local entropy theory},
Ergod. Th. Dynam. Sys. \textbf{29} (2009), 321--356.

\bibitem{G15} F. Garc\'a-Ramos, \textit{Weak forms of topological and measure theoretical equicontinuity:
 relationships with discrete spectrum and sequence entropy},
  Ergod. Th. Dynam. Sys. \textbf{37} (2017), 1211--1237.

\bibitem{J17}  Y. Jiang, \textit{Zero Entropy Interval Maps And MMLS-MMA Property},
preprint, arXiv:1708.05755

\bibitem{K15} D. Karagulyan, \textit{On M\"obius orthogonality for interval maps of zero entropy and
orientation-preserving circle homeomorphisms}, Ark. Mat., \textbf{53} (2015), 317--327.

\bibitem{KL07} D. Kerr and H. Li, \textit{Independence in topoloigcal and C$^*$-dynamics},
Math. Ann. \textbf{338} (2007), 869--926.

\bibitem{KS89} M. Kuchta and J. Sm\'ital, \textit{Two-point scrambled set implies chaos},
European Conference on Iteration Theory (Caldes de Malavella, 1987), Worlds Sci. Publ.,
Teaneck, NJ, 1989, pp. 427--430.

\bibitem{HW}  G. H.  Hardy and  E. M. Wright, \textit{An Introduction to the Theory of Numbers},
Oxford, Clarendon Press, 1979.

\bibitem{H06} W. Huang, \textit{Tame systems and scrambled pairs under an abelian group action},
Ergod. Th. Dynam. Sys. \textbf{26} (2006), 1549--1567.

\bibitem{HWY16}W. Huang, Z. Wang and G. Zhang, \emph{Mobius disjointness for topological models of ergodic systems with discrete spectrum},
preprint, arXiv:1608.08289

\bibitem{HY06} W. Huang and X. Ye , \textit{A local variational relation and applications},
 Israel J. Math. \textbf{151} (2006), 237--80

\bibitem{L11} J. Li, \textit{Chaos and Entropy for Interval Maps},
 J. Dyn. Diff. Equat. \textbf{23} (2011), 333--352.

\bibitem{L16} J. Li, \textit{Measure-theoretic sensitivity via finite partitions,}
 Nonlinearity, \textbf{29} (2016), 2133--2144

 \bibitem{LTY15} J. Li, S. Tu and X. Ye, \textit{Mean equicontinuity and mean sensitivity},
 Ergod. Th. Dynam. Sys. \textbf{35} (2015), 2587--2612.

 \bibitem{LY16} J. Li and X. Ye, \textit{Recent development of chaos theory in topological dynamics},
 Acta Math. Sin. (Engl. Ser.) \textbf{32} (2016), 83--114.

\bibitem{L93} S. Li, \textit{$\omega$-chaos and topological entropy},
Trans. Amer. Math. Soc. \textbf{339} (1993), 243--249.

\bibitem{LY} T. Y. Li and J. A. Yorke, \textit{Period three implies chaos},
 Amer. Math. Monthly \textbf{82} (1975), 985--992.

\bibitem{LM} J. Llibre and M. Misiurewicz, \textit{Horseshoes, entropy and periods for graph maps},
Topology, \textbf{32} (1993), 649--664.

\bibitem{MS07v2} J. Mai and T. Sun, \textit{The $\omega$-limit set of a graph map},
Topology Appl., \textbf{154} (2007), 2306--2311.

\bibitem{S09} P. Sarnak, \textit{Three lectures on the M\"obius function, randomness and dynamics},
 Lecture notes, IAS (2009).

\bibitem{R15} Ruette, S. \textit{Chaos on the interval},
University Lecture Series, 67.
American Mathematical Society, Providence, RI, 2017.

\bibitem{RS14} S. Ruette and L. Snoha,
\textit{ For graph maps, one scrambled pair implies
Li-Yorke chaos,} Proc. Amer. Math. Soc. \textbf{142} (2014), 2087-2100.

\bibitem{T11} F. Tan, \textit{The set of sequence entropies for graph maps},
Topology Appl. \textbf{158} (2011), 533--541.

\bibitem{TYZ10} F. Tan, X. Ye and R. Zhang, \textit{The set of sequence entropies for a given space},
Nonlinearity \textbf{23} (2010), 159--178.

\bibitem{W82} P. Walters, \emph{An Introduction to Ergodic Theory (Graduate Texts in Mathematics, 79)},
 Springer, New York-Berlin, 1982.

 \bibitem{X05} J. Xiong, \textit{Chaos in topological transitive systems},
Sci. China A \textbf{48} (2005), 929--939.

\end{thebibliography}
\end{document}